\documentclass[12pt]{amsart}
\usepackage[dvips]{graphicx}
\usepackage{color}

\def\cal{\mathcal}

\def\RR{\mathbb R}
\def\HH{\mathbb H}
\def\ts{{\cal T}(S)}
\def\a{\alpha}
\def\s{\sigma}
\def\L{{\cal L}}

\def\T{{\cal T}}

\def\da{d_\a(\nu^+,\nu^-)}
\def\dd{\partial}
\def\ep{\epsilon} 

\def\np{{\nu^+}}
\def\nm{{\nu^-}}

\def\Teich{Teichm\"uller }

\def\mul{\stackrel{{}_\ast}{\asymp}}
\def\add{\stackrel{{}_+}{\asymp}}
\def\ladd{\stackrel{{}_+}{\prec}}
\def\gadd{\stackrel{{}_+}{\succ}}
\def\lmul{\stackrel{{}_\ast}{\prec}}
\def\gmul{\stackrel{{}_\ast}{\succ}}

\def\t{\tau}

\def\G{{\cal G}}

\def\TS{{\rm Teich}(S)}
\def\teich{{\cal T}}
 
 \def\ML{\mathop{\cal{ ML}} (S)}
 \def\mod{\mathop{\rm{Mod}}}
   
\def\ext{\rm{Ext}}

\newtheorem{theorem}{Theorem}[section]
\newtheorem{definition}[theorem]{Definition}
\newtheorem{lemma}[theorem]{Lemma}
\newtheorem{proposition}[theorem]{Proposition}
\newtheorem{corollary}[theorem]{Corollary}

\newtheorem{introthm}{Theorem}

\begin{document}
%%%%%%%%%%%%%%%%%%%%%%%%%%%%%%%%%%%%%%%%%%%%%%%%%%%%%%%%%%%%%%%%%%%%%
\title{Lines of minima are uniformly quasi-geodesic}
\begin{abstract}
We continue the comparison between lines of minima and Teichm\"uller geodesics
begun in~\cite{crs}.
 For two measured laminations $\nu^+$ and $\nu^-$ that fill up a
  hyperbolizable surface $S$ and for $t \in (-\infty, \infty)$, let
  $\L_t$ be the unique hyperbolic surface that minimizes the length
  function $e^tl({\nu^+}) + e^{-t} l({\nu^-})$ on \Teich space. 
  We prove that the path $t \mapsto \L_t$ is a Teichm\"uller quasi-geodesic.
\end{abstract}

\author{Young-Eun Choi} 
%\\Department of Mathematics 
%\\ University of California \\Davis CA 95616, USA\\
\email{choiye@psu.edu}
%\urladdr{http://www.math.ucdavis.edu/$\sim$choiye/} 

\author{Kasra Rafi}
%\\Department of Mathematics 
%\\ University of Conneticut at Davis\\Storrs CT 06269, USA
\email{rafi@math.uchicago.edu}
%\urladdr{http://www.math.uconn.edu/$\sim$rafi/} 

\author{Caroline Series}
%\\Mathematics Institute\\ University of Warwick
%\\Coventry CV4 7AL, UK}
\email{C.M.Series@warwick.ac.uk}
%\urladdr{http://www.maths.warwick.ac.uk/$\sim$cms/} 
\maketitle
\date{}

 %%%%%%%%%%%%%%%%%%%%%%%%%%%%%%%%%%%%%%%%%%%%%%%
 \section{Introduction}
\label{sec:introduction}
%%%%%%%%%%%%%%%%%%%%%%%%%%%%%%%%%%%%%%%%%%%%%%% 
This paper continues the comparison between lines of minima and Teichm\"uller geodesics begun in~\cite{crs}. Let $S$ be a  hyperbolizable surface of finite type 
and $\ts$ be the \Teich space of $S$. Let $\nu^+$ and $\nu^-$ be two measured 
laminations that fill up $S$. The  associated \emph{line of minima} is the path 
$ t \mapsto \L_t \in \ts$,  where $\L_t=\L_t(\np,\nm)$ is  the unique hyperbolic 
surface that minimizes the length function $e^tl({\nu^+}) + e^{-t} l({\nu^-})$ on $\ts$, 
see~\cite{klminima} and Section~\ref{sec:linesofm} below.  Lines of minima have significance for hyperbolic $3$-manifolds:  infinitesimally bending $\L_t$ along the lamination $\nu^+$ results in a quasifuchsian group whose convex core boundary  has bending measures in the projective classes   $\nu^+$ and $\nu^-$ and in the ratio $e^{2t}:1$, see~\cite{series}. 
In this paper we prove:
 \begin{introthm}
 \label{thm:qgeo0}
The line of minima $\L_t$, $t\in \mathbb{R}$, is a quasi-geodesic with respect
to the \Teich metric.
In other words, there are universal constants
$c>1,C>0$, depending only on the
topology of $S$, such that 
  for any
$a,b \in \RR$ with $a<b$, we have
$$(b-a)/c -C \leq d_{\ts}(\L_a, \L_b) \leq c(b-a)+C,$$
where $d_{\ts}$ is the \Teich distance. 
\end{introthm}

An obvious way to approach this would be to compare the time-$t$ surface  $\L_t$
with the corresponding surface $\G_t$ on the \Teich geodesic whose horizontal
and vertical foliations at time $t$ are respectively, $e^t \np$ and $e^{-t}\nm$
\cite{G-M}.
In~\cite{crs}, we did just this. We showed that if neither surface $\L_t$
nor  $\G_t$ contains short curves, that is, they are both contained in the thick part of 
\Teich space, then the \Teich distance between them
is bounded above by a uniform constant that is independent of $t$. 
More generally, we showed that the set of curves which are short on the two surfaces coincide.  We also showed, however,  that the ratio of lengths of the same short curve on the two surfaces may be arbitrarily large,
so that  the path $\L_t$ may deviate arbitrarily far from $\G_t$.
It is therefore not immediately obvious how to derive Theorem~\ref{thm:qgeo0} from~\cite{crs}.
To explain our method, we first summarize the results of~\cite{crs}  in more detail.

\bigskip

It turns out that on both $\L_t$
and  $\G_t$, a curve $\a$ is short if and only if at least one of  two quantities  $D_t(\a)$ and $K_t(\a)$ is large. These quantities 
depend  on the topological relationship between $\a$ and  the defining laminations $\nu^+$ and $\nu^-$. 
They relate to 
the modulus of a maximal embedded annulus around $\a$; 
the modulus of a {\em flat} annulus is approximately $D_t(\a)$ and the 
modulus of an {\em expanding} annulus is approximately $\log K_t(\a)$,
see~\cite{minskyharmonic} and Sections~\ref{sec:flat} and~\ref{sec:estimates}
below.
We say that a curve is \emph{extremely short} if it is less than some prescribed $\ep_0>0$ depending only on the topology of $S$,  see Section~\ref{sec:background}. The essential results in~\cite{crs} were the following estimates (see Section~\ref{sec:background}
for notation):
\begin{theorem}[\cite{crs} Theorems 5.10, 5.13, 7.13, 7.14]
\label{thm:GLshort}
Let $\a$ be a simple closed curve on $S$.
If $\a$ is extremely short on $\G_t$ then
$$\frac{1}{l_{\G_t}(\a)} \asymp \max\{ D_t(\a), \log K_{t}(\a) \},$$
while if $\a$ is extremely short on $\L_t$ then
  $$\frac{1}{l_{\L_t}(\a)} \asymp \max \{ D_t(\a), \sqrt{K_t(\a)} \}.$$
\end{theorem}

\begin{theorem} [\cite{crs} Theorem 7.15] 
\label{thm:GLdistance}
The \Teich distance between
$\L_t$ and $\G_t$ is given by
$$d_{\ts}(\L_t,\G_t)  \add \frac{1}{2} \log \max_\a \frac{l_{\G_t}(\a)}{l_{\L_t}(\a)}
,$$
where the maximum is taken over all simple closed curves $\a$ that are extremely short
in $\G_t$. In particular, the distance between the thick parts of $\L_t$ and  $\G_t$
is bounded.
\end{theorem}

It follows from these  results, that along intervals on which either there are no short curves, or on which $D_t (\a)$ dominates for all short curves $\a$, the surfaces $\L_t$ and $\G_t$ remain a bounded distance apart. However  the path $\L_t$ may deviate arbitrarily far from $\G_t$ along time  intervals on which  $K_t(\a)$ is large and dominates $D_t (\a)$.  
The situation is complicated by the fact that as we move along $\L_t$, the family of curves which are short at a given point in time will vary with $t$,  so that the intervals along which different curves $\a$ are short will overlap.

\bigskip

In addition to the above results from~\cite{crs}, there are two main ingredients in the proof of 
 Theorem~\ref{thm:qgeo0}. The first is a detailed comparison of  the rates of change of $K_t (\a)$ and $D_t (\a)$ with $t$. Some simple estimates are made in Lemmas~\ref{lem:Kdecay}  and~\ref{lem:DandK}, with more elaborate consequences drawn   in Lemma~\ref{lem:HLdistance}  and especially Lemma~\ref{lem:dichotomy}. 
These results use Minsky's product regions theorem (see Theorem~\ref{thm:minskyproduct}), which allows us to  reduce calculations of distance in regions of  \Teich space in which a given family of curves is short, to straightforward estimates in $\HH^2$.
To apply  Minsky's theorem, we need not only to compare lengths but also twists. We rely on the bounds on twists  proved in~\cite{crs} and reviewed in Theorems~\ref{thm:gtwist} and~\ref{thm:Ltwist}; these enter in a crucial way into the proof of Lemma~\ref{lem:HLdistance}.

The second main ingredient is control of distance along intervals 
along which $K_t(\a)$ is large. Consider  the surface $S_\a$ obtained by cutting $S$ along a short curve $\a$ and replacing the two resulting boundary components by punctures. 
The  following rather surprising result, proved in Section 4, states that on intervals along which $K_t(\a)$ is large, we can estimate the \Teich distance by restricting to the \Teich space of  the surface 
 $S_\a$. In other words, the contribution to \Teich distance in Minsky's formula~\ref{thm:minskyproduct} due to the short curve $\a$ itself may be neglected, see Theorem~\ref{thm:Klarge} for a precise statement.
\begin{introthm}
\label{thm:Klarge0}
If $K_t(\a)$ is sufficiently large for all $t \in [a,b]$, the distance 
in $\T(S_\a)$ between the restrictions of $\G_a$ and $\G_b$ to
$S_\a$ is equal to $b-a$, up to an additive error that is bounded by a constant
depending only on the topology of $S$. 
\end{introthm}

The proof of Theorem~\ref{thm:qgeo0} requires estimating upper and lower bounds for
$d_{\T(S)}(\L_a,\L_b)$ over very large time intervals $[a,b]$. Given the first of the two ingredients above, the upper bound is relatively straightforward. The lower bound  depends on Theorem~\ref{thm:Klarge0}. The actual application involves a rather subtle inductive procedure based on Lemma~\ref{lem:dichotomy} which shows that at least one term in Minsky's formula~\ref{thm:minskyproduct} 
always involves a contribution comparable to $b-a$.

 \bigskip
   
The paper is organized as follows. Section 2 gives standard background and introduces the twist $tw_{\sigma} (\xi,\a)$ of a lamination $\xi$ about a curve $\a$ with respect to a hyperbolic metric $\sigma$. We give the main estimates about twists from~\cite{crs}.
In Section 3,
we recall from \cite{crs} the definitions of   $D_t(\a)$ and $K_t(\a)$
and derive some elementary results about their rates of change with $t$.
In Section 4 we prove Theorem~\ref{thm:Klarge0}
and in Section 5 we prove Theorem~\ref{thm:qgeo0}.

\subsection*{Acknowledgment} We would like to thank the referee for helpful comments.
\section{Background}
\label{sec:background}
 %%%%%%%%%%%%%%%%%%%%%%%
\subsection*{Notation}
%%%%%%%%%%%%%%%%%%%%%%%
Since we will be dealing mainly with coarse estimates, we want to
avoid heavy notation and keep track of constants which are universal, in that they do
not depend on any specific metric or curve under discussion.  For
functions $f,g$ we write $f\asymp g$ to mean
 that there are constants $c \geq 1,C\geq 0$, depending only on the
topology of $S$ and the fixed constant $\ep_0$ (see below), such that 
$$\frac{1}{c} g(x) - C \leq f(x) \leq c g(x) +C.$$
We use  $f \mul g$ and $f\add g$ to mean that these inequalities hold
 with  $C=0$ and $c=1$, respectively.
The symbols  ${\prec}$,
$\stackrel{{}_+}{\prec}$, $\stackrel{{}_\ast}{\prec}$, etc.,
are defined similarly.  
In particular, we write
 $X \prec 1$ to indicate $X$ is bounded above by a positive
constant depending only on the topology of $S$ and $\ep_0$.

%For a positive quantity $X$, we often write
%$X=O(1)$ instead of $X \prec 1$ to indicate $X$ is bounded above by a
%constant depending only on the topology of $S$ and $\ep_0$, and more
%generally we write $X = O(Y)$ to mean that $X/Y = O(1)$ for a positive
%function $Y$.

%%%%%%%%%%%%%%%%%%
\subsection*{Short curves}
%%%%%%%%%%%%%%%%%%%
Let $\cal C (S)$ denote the set of isotopy classes of non-trivial, non-peripheral
simple closed curves on $S$.
The length of the
geodesic representative of $\a \in \cal C (S)$ 
with respect to a hyperbolic metric $\s \in \ts$ will be denoted $l_\s(\a)$.
In our dealings with short curves we will have to make various assumptions to ensure the validity of our estimates, which all require that the length $l_\s(\a)$ of a `short' curve be less than various constants, in particular 
less than the Margulis constant. We suppose that $\ep_0>0$ is chosen 
once and for all to satisfy all needed assumptions, and  
say a simple closed curve $\a$ is 
\emph{extremely short} in $\s$ if $l_{\sigma}(\alpha) < \ep_0$.

 \subsection*{Measured laminations and \Teich space.}
We denote the space of measured laminations on $S$ by $\ML$ and  write $l_\s(\xi)$ for the hyperbolic length of
a measured lamination $\xi \in \ML$.  For $\xi \in
\ML$, we denote the underlying leaves by $|\xi|$.

\subsection*{Kerckhoff lines of minima}
\label{sec:linesofm}
Suppose that $\np,\nm \in \ML$ fill up $S$, meaning that  the
sum of (geometric) intersections 
$i(\np,\xi) + i(\nm,\xi) >0$ for all $\xi \in \ML$.  Kerckhoff~\cite{klminima}
showed that the sum of length functions
$$\sigma \mapsto l_\s(\nu^{+}) +l_\s(\nu^{-})$$ 
has a  unique global minimum on $\teich(S)$.
Moreover, as $t$ varies in $(-\infty,\infty)$, the minimum $\L_t \in
\teich(S)$ of $l(\nu_t^+)+l(\nu_t^-)$ for the measured laminations
$\nu^{+}_t= e^{t}\np$ and $\nu^{-}_t= e^{-t}\nm$ varies continuously
with $t$ and traces out a  path $t \mapsto \L_t$ called the {\em line
  of minima $\L(\np,\nm)$ of }$\nu^{\pm}$.

%%%%%%%%%%%%%%%%%%%%%%
\subsection*{Teichm\"uller geodesics}
%%%%%%%%%%%%%%%%%%%%%%%
 A pair of laminations $\np,\nm \in \ML$ which fill up $S$ also defines a  \Teich geodesic $   \G= \G(\np,\nm)$.  The time-$t$ surface $\G_t \in \G$ is the unique Riemann surface
that supports a
quadratic differential $q_t$ whose horizontal and vertical foliations are the measured foliations corresponding to  $\nu_t^+$ and $\nu_t^-$ respectively, see \cite{G-M}, \cite{levitt}. 
Flowing distance $d$ along $\G$  expands
the vertical foliation by a factor $e^{d}$ and contracts
the horizontal foliation by $e^{-d}$. 
By abuse of notation, we
denote the hyperbolic metric on the surface $\G_t $ also by $\G_t$, 
and likewise denote the quadratic differential metric defined by $q_t$ also by $q_t$.
%%%%%%%%%%%%%%
\subsection*{Balance time}
For a curve $\a \in \cal C(S)$ that is neither a component of the 
vertical nor the horizontal foliation, let $t_{\a}$ denote the
\emph{balance time} of $\a$ at which $i(\a,\nu^+_t) = i(\a,\nu^-_t)$. 
Along $ \G$, 
a curve is shortest near its balance time. More precisely, we have
the following proposition which follows 
from  ~\cite{rafi2} Theorem 3.1:
\begin{proposition}
\label{prop:aroundbalancetime}
Choose $\ep>0$ so that $\ep< \ep_0$ and suppose that $l_{\G_{t_\a}}(\a) < \ep.$
Let $I_\a = I_\a(\ep)$ be the maximal connected interval containing $t_\a$ such that
$l_{\G_t}(\a) < \ep$ for all $t \in I_\a$. Then there is a constant $\ep'>0$ depending
only on $\ep$ such that $l_{\G_t}(\a) \geq \ep'$ for all $t \notin I_\a$.  (If  $l_{\G_{t_\a}}(\a) \geq \ep$ then set  $I_\a = \emptyset$.)
\end{proposition}

Curves which are components of the vertical foliation ($ i(\a,\nu^-)= 0$, called 
{\em vertical}) or the 
 horizontal  foliation ($ i(\a,\nu^+)= 0$, called {\em horizontal}) are exceptional but in general easier to handle.  In such cases, $t_{\a}$ is undefined. However, for
reasons of continuity, it is natural to adopt the convention that when $\a$ is vertical $t_\a=
 -\infty$ and when $\a$ is horizontal $t_\a = \infty$.
Moreover, the arguments used to prove Proposition~\ref{prop:aroundbalancetime} still hold;  when $\a$ is vertical (resp. horizontal), we define $I_{\a} = (-\infty, c)$ 
(resp. $I_{\a} = (d,\infty)$) to be the maximal interval
where $l_{\G_t}(\a) < \ep$.

  \subsection*{Flat and expanding annuli}
  \label{sec:flat}
  Let $\s$ be a hyperbolic metric and let $q$ be any quadratic differential metric
in the same conformal class.
Let $A$ be an annulus in $(S,q)$ with piecewise smooth boundary.
 The following notions are due to Minsky \cite{minskyharmonic}.
 We say $A$ is {\em regular} if the boundary components $\dd_0,\dd_1$
 are equidistant from one another and the curvature along
 $\dd_0,\dd_1$ is either non-positive at every point or non-negative
 at every point (see \cite{minskyharmonic} or \cite{crs} for details).
 We follow the sign convention that the
curvature at a smooth point of $\dd A$ is positive if the acceleration vector points
into $A$. Suppose $A$ is a regular annulus such that the total curvature of
$\dd_0$ satisfies
$\kappa(\dd_0)\leq 0$. Then, it follows from the 
Gauss-Bonnet theorem that $\kappa(\dd_1) \geq 0$.
We say $A$ is {\em flat} if $\kappa(\dd_0)=\kappa(\dd_1)=0$ and
 say $A$ is {\em expanding} if $\kappa(\dd_0) <0$,
and call $\dd_0$ the {\em inner} boundary and $\dd_1$ the {\em outer}
boundary.

A regular annulus is {\em primitive} if it contains no singularities
of $q$ in its interior.
It follows that a flat annulus is primitive and is isometric to a cylinder
obtained from a Euclidean rectangle by identifying one pair of parallel sides.
An expanding annulus that is primitive is 
coarsely isometric to an annulus bounded by two concentric circles in the plane.

The length of a curve $\a$ which is short in $(S,\s)$ can be estimated 
by the modulus of a primitive annulus around it:
\begin{theorem}[\cite{minskyharmonic} Theorem 4.5, \cite{crs} Theorem 5.3]
\label{thm:basic}
Suppose $\a \in \cal C (S)$ is extremely short in $(S,\s)$. Then for any
quadratic differential metric $q$ in the same conformal class as  $\s$, there is
an annulus $A$ that is primitive with respect to $q$ whose core is homotopic to
$\a$ such that
$$\frac{1}{l_\s(\a)} \asymp \mod (A).$$
\end{theorem}
Furthermore, the modulus of a primitive annulus is estimated as follows: 
\begin{theorem}[\cite{minskyharmonic} Theorem~4.5, \cite{rafi1}  Lemma~3.6] 
\label{thm:modcomparison}Let $A \subset S$ be a primitive
annulus.
Let $d$ be the $q$-distance between the boundary components $\dd_0,
\dd_1$. If $A$ is expanding let $\dd_0$ be the inner boundary. Then either
\begin{enumerate}
\item[(i)]  $A$ is flat and $\mod A = d/l_q(\dd_0)=d/l_q(\dd_1)$ or  
\item[(ii)] $A$ is expanding and $\mod A \asymp \log [d/l_q(\dd_0)]$.
\end{enumerate}
\end{theorem}

%%%%%%%%%%%%%%%%%%%%%%%%
\subsection*{Minsky's product regions theorem}
\label{sec:minskyproduct}
%%%%%%%%%%%%%%%%%%%%%%%%
Our main tool  for estimating \Teich distance 
 is Minsky's product regions theorem, which reduces the estimation of the distance
between two surfaces on which a given set $\Gamma$ of curves is short to a calculation in the hyperbolic plane $\HH^2$. 
To give a precise statement, we introduce the following notation.
Choose a pants curves system on $S$ that contains $\Gamma$, and 
for a curve $\a$ in the pants system let $s_{\a}(\s)$ be the 
Fenchel-Nielsen twist coordinate of $\a$. 
(Here $s_{\a}(\s) = \tilde s_\a(\s) / l_{\s}(\a)$, where $\tilde s_\a(\s)$ is the  actual hyperbolic distance twisted round $\a$, see Minsky~\cite{minskyproduct} for details.)
Let  $ \T_{thin}({\Gamma},\ep_0)$ be the subset of $\ts$ on which 
all the curves in $\Gamma$ have length less than $\ep_0$ and let
$S_{\Gamma}$ be the analytically finite surface obtained from $S$ by pinching
all the curves in $\Gamma$. By forgetting the Fenchel-Nielsen length and twist
coordinates associated to the curves in $\Gamma$ but retaining all remaining
Fenchel-Nielsen coordinates, we obtain a projection
$\Pi_\Gamma: \ts \to \T(S_{\Gamma}) $. For each $\a \in \Gamma$, 
 let $\HH_{\a}$ denote a copy of the
upper-half plane and let $d_{\HH_\a}$ denote \emph{half} the 
usual hyperbolic metric on $\HH_\a$ (see Lemma 2.2 in~\cite{minskyproduct} for the factor). Define 
$\Pi_{\a}: \ts \to
\HH_{\a}$ by $\Pi_{\a}(\s) = s_{\a}(\s) + i/ l_{\s}(\a) \in \HH_{\a}$.
Then the product regions theorem states:
\begin{theorem} [Minsky \cite{minskyproduct}]
 \label{thm:minskyproduct} 
 Let $\s,\t \in \T_{thin}({\Gamma},\ep_0) $.  Then 
$$d_{\ts}(\s,\t) \add
 \max_{\a \in {\Gamma}} \{ d_{\T(S_{\Gamma})}(\Pi_\Gamma(\s),\Pi_\Gamma(\t)),
 d_{\HH_{\a}}(\Pi_{\a}(\s), \Pi_{\a}(\t)) \}.$$
\end{theorem}
\noindent To simplify notation, we write $d_{\HH_{\a}}(\s,\tau)$ instead of
$d_{\HH_{\a}}(\Pi_{\a}(\s), \Pi_{\a}(\tau))$ and $ d_{\T(S_{\Gamma})}(\s,\t) $ 
instead of $d_{\T(S_{\Gamma})}(\Pi_\Gamma(\s),\Pi_\Gamma(\t))$.

\medskip

In practice, we usually apply Minsky's theorem with the aid of the following estimate 
from geometry in $\HH^2$.
The hyperbolic distance between two points $z_1,z_2$ in
$\HH^2$ is given by
 $$\cosh 2d_{\HH}(z_1,z_2) = 1 + \frac{|z_1-z_2|^2}{2 \, \mbox{Im}
  \,z_1 \mbox{Im}\, z_2} .$$
Let $\s_a, \s_b$ be two points in $\TS$ at which a curve $\a$ is short. Let
$\ell_a, \ell_b$ and $s_a, s_b$
denote  
the Fenchel-Nielsen twist coordinate of $\a$ at $\s_a, \s_b$ respectively.
It follows easily from the above formula that
\begin{equation}
\label{eqn:distance-in-H2}
 d_{\HH^2_\a}(\s_a, \s_b)\add \frac{1}{2} \log \max \Big\{ 
|s_a - s_b|^2 \ell_a \ell_b, \frac{\ell_a}{\ell_b},
 \frac{\ell_b}{\ell_a} \Big\}. 
\end{equation}

\subsection*{Twists}

Our estimates also require taking account of  the \emph{twist}  $tw_\s(\nu,\a)$ of a lamination $\nu$ round $\alpha$ with respect to a hyperbolic metric $\s$.  Following Minsky, we define
$$
tw_\s(\nu,\a) = \inf \frac{\tilde s}{l_\s(\a)},$$
where $\tilde s$ is the signed hyperbolic distance between the perpendicular projections of the endpoints of a lift of a geodesic in $|\nu|$ at infinity onto a lift of $\alpha$, and the infimum is over all lifts of leaves of 
$|\nu|$ which intersect $\alpha$, see~\cite{minskyproduct} or~\cite{crs}
for details. 
We write $Tw_\s(\nu,\a)$ for $|tw_\s(\nu,\a)|$.
Notice that the twist $tw_\s(\nu,\a)$ does not depend on the measure on $\nu$,
only on the underlying lamination
$| \nu|$.

The  twist is closely related to the Fenchel-Nielsen twist coordinate.
Specifically, we have:
\begin{lemma} [Minsky \cite{minskyproduct} Lemma 3.5] 
\label{lem:minskytwist}
For any lamination $\nu \in \ML$ and any two metrics $\s, \s' \in
\ts$,
$$ | (tw_\s(\nu,\a) - tw_{\s'}(\nu,\a)) -
(s_\a (\s)  -
 s_{\a} (\s'))| \leq 4.$$ 
\end{lemma}

Although   $tw_\s(\nu,\a)$ depends on the metric $\s$,  for  $\nu_1,\nu_2 \in \ML$ the difference $tw_\s(\nu_1,\a) -
tw_\s(\nu_2,\a)$   is independent of $\s$ up to a universal additive constant, see~\cite{minskyproduct} and~\cite{crs} Section 4. This motivates the following definition:
 \begin{definition} For $\a \in {\cal C}(S)$ and $\nu_1, \nu_2 \in \ML$, 
 the relative twist of $\nu_1$ and $\nu_2$ round $\a$ is
\label{defn:algint}
$$ d_\a(\nu_1,\nu_2) = \inf_\s |tw_\s(\nu_1,\a) -
tw_\s(\nu_2,\a)|,$$
where the infimum is taken over all 
hyperbolic metrics $\s \in \ts$.
\end{definition}
(The relative twist $
d_\a(\nu_1,\nu_2) $ agrees up to an additive constant with the definition of
subsurface distance between the projections of $|\nu_1|$ and $|\nu_2|$
to the annular cover of $S$ with core $\a$, as defined
in~\cite{masurminskyII} Section 2.4 and used throughout~\cite{rafi1,
  rafi2}.)

  \medskip
  Rafi~\cite{rafi2}, see also~\cite{crs} Section 5.4,  introduced a similar notion of the twist    $ tw_q(\nu,\a)$ with respect to a quadratic differential metric $q$ compatible with $\sigma$ and proved the following result which enters into  the proof of Theorem~\ref{thm:Klarge}:
  \begin{proposition}[\cite{rafi2} Theorem 4.3 and \cite{crs} Proposition 5.7] 
 \label{prop:twistcomparison} Suppose that
 $\s \in \ts$ is a hyperbolic metric and
 $q$ is a compatible quadratic differential metric.
 For any geodesic lamination $\xi$ 
 intersecting $\a$, we have
 $$
 |tw_{\s}(\xi,\a) - tw_q(\xi,\a)| \prec \frac{1}{l_{\s}(\a)}.$$
\end{proposition}

\medskip 

We shall also need the following important estimates of the twist which  complement Theorem~\ref{thm:GLshort}. If $\a$ is vertical or horizontal, $t_\a$ is defined 
using the convention discussed following Proposition~\ref{prop:aroundbalancetime}.
\begin{theorem}[\cite{crs} Theorems 5.11, 5.13]
\label{thm:gtwist}
Let $\a$ be a simple closed curve on $S$.
If $\a$ is extremely short on $\G_t$ then: 
\begin{eqnarray}
Tw_{\G_t}(\np,\a) &\prec&
\frac{1}{l_{\G_t}(\a)}   \ \mbox{ if }
t > t_\a, \nonumber \\
Tw_{\G_t}(\nm,\a) &\prec&
 \frac{1}{l_{\G_t}(\a)}   \ \mbox{ if }
t < t_\a. \nonumber
\end{eqnarray}
\end{theorem}
\begin{theorem}[\cite{crs} Theorems 6.2, 6.9]
 \label{thm:Ltwist} 
Let $\a$ be a simple closed curve on $S$.
If $\a$ is extremely short on $\L_t$ then,
\begin{eqnarray}
Tw_{\L_t}(\nu^+,\a) &\prec&   \frac{1}{ l_{\L_t}(\a)}  
\ \mbox{ if } t > t_\a , \nonumber \\
Tw_{\L_t}(\nu^-,\a) &\prec&
  \frac{1}{ l_{\L_t}(\a)} 
\ \mbox{ if } t< t_\a . \nonumber
\end{eqnarray} 
\end{theorem}

%%%%%%%%%%%%%%%%%%%%%%%%%%%%%%%%%%%%%%%%%%%%%%

\section{The  length estimates}
\label{sec:estimates}
In this section we discuss the quantities $D_t(\a)$ and $K_t(\a)$ that
appear in the length estimates
in Theorem~\ref{thm:GLshort}.

Let $q_t$ be the unit-area quadratic differential metric on $\G_t$ whose
vertical and horizontal foliations are $e^t \np$ and $e^{-t} \nm$, respectively.
In light of Theorems~\ref{thm:basic} and~\ref{thm:modcomparison}, to
estimate the length of a curve $\a$ which is extremely short in
$\G_t$, it is sufficient to estimate the modulus of a maximal flat or expanding
annulus around $\a$ in $q_t$. The union of  all $q_t$-geodesic representatives of $\a$  foliate a Euclidean cylinder  $F_t(\a)$, which is the maximal flat annulus
whose core is homotopic to $\a$. (The cylinder  is degenerate if the representative of $\a$ is unique.)
On either side of $F_t(\a)$ is attached a maximal expanding annulus. Let
$E_t(\a)$ be the one of larger modulus. 
Up to coarse equivalence, $D_t(\a)$ will be the modulus of
$F_t(\a)$ while   $\log K_t(\a)$  will be the modulus of 
 $E_t(\a)$.

  The precise definition of $D_t(\a)$ is as follows. If $\a$ is not a component of $|\nu^{\pm}|$, we define \begin{equation}
\label{eqn:modF}
D_t(\a) = e^{-2|t-t_{\a}|} d_{\a}(\np,\nm),
\end{equation}
 where
$\da$ is the relative twisting of $\np$ and $\nm$ about $\a$ as 
 defined above.  If  $\a$ is vertical, define $D_t(\a) = e^{-2t} \mod F_0(\a)$
 and  if $\a$ is horizontal, define $D_t(\a) = e^{2t} \mod F_0(\a)$,  where $F_0(\a)$ is the annulus at time $t = 0$. 
 
 The precise definition of $K_t(\a)$ is:
 \begin{equation}
\label{eqn:modE}
K_t(\a) = \frac{d_{q_t}}{l_{q_t}(\dd_0)}, 
\end{equation}
where  $\dd_0$ is the inner
boundary of $E_t(\a)$ and $d_{q_t}$ is the $q_t$--distance between the
inner and outer boundaries of $E_t(\a)$.

The connection with the definition  of  $K_t(\a)$ in~\cite{crs}, and the reasons why 
$D_t(\a)$ and $K_t(\a)$ are coarsely the moduli of $F_t(\a)$ and $E_t(\a)$ respectively, are explained at the end of this section.
The estimate for $1/l_{\G_t}(\a)$ in Theorem~\ref{thm:GLshort} follows
easily from the above definitions and Minsky's  estimates. 
The estimate for $1/l_{\L_t}(\a)$ in the same theorem required a lengthy  separate
analysis.
The only features of these definitions which will concern us here are the estimates in Theorem~\ref{thm:GLshort}, and the relative rates of change of $D_t(\a)$  and $K_t(\a)$ with  time.
 
\subsection*{The rate of change of $D_t(\a)$ and $K_t(\a)$}
The rate of change of $D_t(\a)$  with time is immediate from~(\ref{eqn:modF}).
To estimate the rate of change of $K_t(\a)$ note that, 
since $E_t(\a)$ is maximal, $d_{q_t}$ in Equation~(\ref{eqn:modE}) is half the $q_t$--length of an
essential arc from $\a$ to itself. Since the $q_t$--length of such an arc or a
simple closed curve can increase or decrease at the rate of at most $e^{\pm t}$, 
Equation~\eqref{eqn:modE} implies that $\sqrt{K_t(\a)}$  changes (in the coarse
sense) at a rate  at most $e^t$. More precisely, if $K_t(\a)$ is sufficiently large for all $t \in [a, b]$, 
then
\begin{equation}
\label{eqn:exponential}
e^{-2(b-a)} K_b(\a) \lmul K_a(\a) \lmul e^{2(b-a)} K_b(\a).
\end{equation}
 
In combination with 
Equation~\eqref{eqn:modF} and Theorem~\ref{thm:GLshort}, it follows that
the length of a short curve along $\L$ or $\G$ changes at rate
 at most $e^{2t}$. More detailed control is given by the following two lemmas, which should be understood with our convention on $t_{\a}$ to include the case when $\a$ is vertical or horizontal.   The first shows that  $K_t(\a)$ decays as $t$ moves away from $t_\a$ while the second,  illustrated schematically in Figure~\ref{fig:graphs}, compares rates of change of $D_t(\a)$ and $\sqrt{K_t(\a)}$. 

 \begin{lemma}
The function $K_t(\a)$ decays as $t$ moves away from $t_\a$.  More
 precisely,
\label{lem:Kdecay}
\begin{itemize}
\item[(i)] If $t_\a < v < w$, then $K_v(\a) \stackrel{{}_\ast}{\succ}
K_w(\a)$.
\item[(ii)] If $v < w < t_\a$, then $K_w(\a) \stackrel{{}_\ast}{\succ}
  K_v(\a)$.
\end{itemize}
\end{lemma}
\begin{proof} Suppose first that $\a$ is not a component of  $|\nu^{\pm}|$.
By Lemma 2.1 in~\cite{cr}, see also~\cite{rafi2} Theorem 2.1,
we have $l_{q_t}(\a) \mul e^{|t-t_\a|}l_{q_{t_\a}}(\a)$ for any $t
  \in \RR$. On the other hand, the length of any curve or arc can
 increase or decrease by a factor of at most $e^{\pm t}$. Hence, if $t_\a < v <
  w$, then
$$ K_v(\a) = \frac{d_{q_v}}{l_{q_v}(\a)} \mul
\frac{d_{q_v}}{e^{(v-w)}l_{q_w}(\a)}=
\frac{e^{(w-v)}d_{q_v}}{l_{q_w}(\a)} \geq
\frac{d_{q_w}}{l_{q_w}(\a)} = K_w(\a).
$$
A similar argument can be applied in the case when
$v < w < t_\a$.

If $\a$ is  vertical, then $ l_{q_t}(\a) \mul e^{t}l_{q_{0}}(\a)$, while if it is horizontal 
$  l_{q_t}(\a) \mul e^{-t}l_{q_{0}}(\a)$. The result then follows in the same way.
\end{proof}

 \begin{lemma}
\label{lem:DandK}
  Let $I_\a$ be as in Proposition~\ref{prop:aroundbalancetime} and
  let $[a,b] \subset I_\a$. Suppose that
  $D_u(\a) = \sqrt{K_u(\a)}$ for some $u \in [a,b]$.
\begin{itemize}
\item[(i)] If  $t_\a < u$, then $\sqrt{K_t(\a)} \gmul D_t(\a)$ for all
$t \in [u,b].$
\item[(ii)] If  $u< t_\a$, then $\sqrt{K_t(\a)} \gmul D_t(\a)$ for all
$t \in [a,u].$
\end{itemize}
\end{lemma}
\begin{proof}
We refer to Figure~\ref{fig:graphs} for a schematic picture of the
two graphs.   The proof is based on the fact that $\sqrt{K_t(\a)}$
  decays at a slower rate than $D_t(\a)$ as $t$ moves away from
  $t_\a$. If $t_\a < u$, then for any $t > u$
  we have
$$
K_t(\a)= \frac{d_{q_t}}{l_{q_t}(\a)}
\gmul \frac{e^{-(t-u)} d_{q_u}}{e^{(t-u)} l_{q_u}(\a)}
= e^{-2(t-u)} K_u(\a).$$
Therefore, 
$$\sqrt{K_t(\a)} \gmul e^{-(t-u)} D_u(\a) = e^{(t-u)} D_t(\a) \geq D_t(\a).$$

%%%%%%%%%%%%%%%%%%%%
\begin{figure}[tb]
\begin{center}
\centerline{\includegraphics{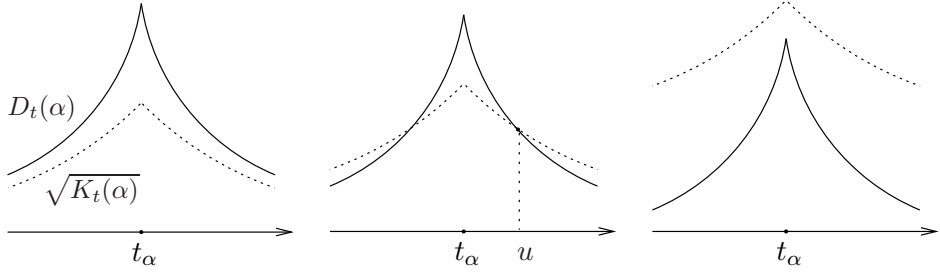}}
%\centerline{\includegraphics{figure1.pdf}}
\caption{Schematic graphs of $D_t(\a)$ and $\sqrt{K_t(\a)}$.
The function $D_t(\a)$ changes at  rate $e^{2t}$ while 
$\sqrt{K_t(\a)}$ changes at rate at most $e^{t}$. }
 \label{fig:graphs}
\end{center}
\end{figure}
%%%%%%%%%%%%%%%%

A similar argument can be applied in the case when
$u < t_\a$.
\end{proof}

\subsection*{Alternative definitions of $D_t(\a)$ and $K_t(\a)$.}
\label{sec:option} 

The remarks which follow, which may be helpful in clarifying background from~\cite{crs},  are not essential for the proof of Theorem~\ref{thm:qgeo0}.
  
 The claim that $D_t(\a)$  is coarsely equal to the modulus of
$F_t(\a)$ is justified by~\cite{crs} Proposition 5.8  (section 5.6 for the exceptional case) which states  that
 $\text{Mod}F_t(\a) \asymp D_t(\a)$.
The proof is an exercise in Euclidean geometry, combined with Rafi's  comparison Proposition~\ref{prop:twistcomparison}  between the twist in the quadratic and hyperbolic metrics. For example, at the balance time
 $  t_{\a}$, the horizontal and vertical leaves both make an angle $\pi/4$ with the
  $q_{t_\a}$-geodesic representatives of $\alpha$. In this case, a
   leaf of $\nu_{t_\a}^+$ or $\nu_{t_\a}^-$ intersects $\eta$
   approximately (up to an error of $1$) $l_{q_{t_\a}}(\eta) /
   l_{q_{t_\a}}(\a)$ times, so the modulus of $F_{t_\a}(\a)$ is
   approximated by $Tw_{F_{t_\a}}(\np,\a)= Tw_{F_{t_\a}}(\nm,\a)$, where $tw_{F_{t_\a}}$ (and $Tw_{F_{t_\a}}$) means the twist in the $q$-metric restricted to $F_{t_\a}$.   The result would follow on noting that $tw_{F_{t_\a}}(\np,\a)$ and $ tw_{F_{t_\a}}(\nm,\a)$ have opposite signs, except that $d_{\a}$ involves  hyperbolic twists on $S$ rather than  $q$-twists in $F_{t_\a}$. This is resolved using Proposition~\ref{prop:twistcomparison}, see~\cite{crs}  for further details. 
 
That $\log K_t(\a)$ is coarsely the modulus of $E_t(\a)$ follows from
Theorem~\ref{thm:modcomparison}.
The above is not the definition of $K_t(\a)$ given  in~\cite{crs}, but it is coarsely equivalent.
Specifically,
let $Y_1,Y_2$ be the 
(possibly coincident) thick components adjacent to $\a$ in the thick-thin
decomposition of the hyperbolic metric
$\G_t$.  Set
$$J_t(\a) = \frac{1}{l_{q_t}(\a)} \max  \{
 \lambda_{Y_1}, \lambda_{Y_2}
 \}
 $$
where $\lambda_{Y_i}$ is the length of the shortest non-trivial
non-peripheral simple closed curve on $Y_i$ with respect to the
metric $q_t$. (If either  $Y_i$ is a pair of pants there is a slightly different definition, see~\cite{crs}.)
In~\cite{crs}, we took the above expression for $J_t(\a)$ as the definition of $K_t(\a)$. 
 Proposition 5.9 in~\cite{crs} shows  that if $J_t(\a)$ is sufficiently large, then $J_t(\a)\mul K_t(\a)$
 with  $K_t(\a)$ defined as in (\ref{eqn:modE}) above.

%%%%%%%%%%%%%%%%%%%%%%%%%%%%%%%%%%%%%%%%%%%%%%%%%%%%%%%%%%%%%%%
\section{Expanding annuli that persist}
\label{sec:expanding}
%%%%%%%%%%%%%%%%%%%%%%%%%%%%%%%%%%%%%%%%%%%%%%%%%%%%%%%%%%%%%%

It follows from Theorems~\ref{thm:GLshort} and~\ref{thm:GLdistance}
that if
$D_t(\a) \geq \sqrt{K_t(\a)}$  for every $\a$ that is short in $\G_t$, 
then the distance $d_{\ts}(\G_t,\L_t)$ 
is uniformly bounded. And, if 
$\G_t$ is in the thick part of 
\Teich space,  then $\L_t$ is too, so that on such intervals
$\L_t$ is quasi-geodesic. Thus our attention is focused on 
 time intervals
along which $K_t(\a)$ is large. This is handled with the  following more precise version of Theorem~\ref{thm:Klarge0}: 
\begin{theorem}
\label{thm:Klarge}
Choose $M>0$ to be a constant such that if $K_t(\a) > M$ then $\a$ is
extremely short in $\G_t$. 
(This is possible due to Theorem~\ref{thm:GLshort}.) Suppose that $K_t(\a) > M$ for all $t \in [a,b]$. 
Then
$$
d_{\T (S_\a)} (\G_a, \G_b) \add b-a .$$
\end{theorem}
\begin{corollary} Let $\Gamma$ be a family of disjoint curves on $S$ such that
$K_t(\a) > M$  for all $t \in [a,b]$ and for every $\a \in \Gamma$.
Then
$$
d_{\T (S_{\Gamma})} (\G_a, \G_b)  \add b-a.$$ 
\end{corollary}
\begin{proof} We prove the statement of the theorem; the corollary is immediate.
  The idea is that for each $t \in [a,b]$, we cut  the maximal flat annulus around $\a$ in $(S,q_t)$ out of $S$ and reglue the two boundary components, obtaining a new surface $\overline \G_t$, see Figure~\ref{fig:cutandreglue}. The surfaces $\overline \G_t$ will also move along a \Teich geodesic.  
  In particular  $d_{\T(S)}(\overline \G_a,\overline \G_b)=b-a$.
  On the other hand,  $\overline \G_t$ contains  the same  expanding cylinders round $\a$ as $ \G_t$ so that   
  $K_{\overline \G_t}(\a) = K_t(\a)$. Consideration of the rate of change of $K_t$ with time shows that the contribution to the
   change in  \Teich distance between $\overline \G_a$ and $\overline \G_b$ from the expanding cylinders is on the order of $\log (b-a)$, so that the actual  distance $b-a$ must be realized  due to changes in $\T(S_{\a})$.

 In more detail, this works as follows.
 Let $F=F_a(\a)$ be the maximal flat annulus around $\a$ in $(S,q_a)$.  
   The arcs in $F$ that are perpendicular to $\dd F$ define an
  isometry $f$ from one component of $\dd F$ to the other. Let
  $\overline \G_a$ be the surface obtained by removing $F$ and gluing
  the components of $\dd F$ together via $f$  (also making sure to preserve the marking), see the upper two surfaces in Figure~\ref{fig:cutandreglue}. Let $\overline \a$ be
  the gluing curve in $\overline \G_a$.
  Since the vertical and horizontal foliations of $q_a$ match along
  $\overline \a$, the surface $\overline \G_a$ is naturally equipped with
  vertical and horizontal foliations $\overline \nu_a^\pm$ and quadratic
  differential $\overline q_a$, which we assume is scaled to have area one.
  Let $\{\overline \G_t\}$ be the \Teich geodesic corresponding to $\overline
  q_a$ and let $\overline q_t$ be the corresponding family of quadratic
  differentials.  Then $d_{\T(S)}(\overline \G_a,\overline \G_b)=b-a$.
  Observe that the surface $\overline \G_t$ is obtained from 
  $ \G_t$ by cutting the maximal flat annulus $F_t=F_t(\a)$.  Thus,
for each $t$, we have a natural map  $\varphi_t: (S,q_t)\setminus F_t
  \to (S,\overline q_t) \setminus \overline \a$ which fixes points but  scales the metric. 
   Hence, $K_{\overline q_t}(\a)=K_t(\a) > M$ on $[a,b]$ and therefore
  $\a$ is also extremely short in $\overline \G_t$ on $[a,b]$.  
Applying
  Theorem~\ref{thm:minskyproduct}, we get
\begin{equation}
\label{eqn:product}
b-a = d_{\ts}(\overline \G_a,\overline \G_b) \add
\max \left\{ d_{\T(S_\a)}(\overline \G_a,\overline \G_b), 
d_{\HH^2_\a}(\overline \G_a,\overline \G_b) \right\}.
\end{equation}

To prove the theorem, it will suffice to establish the following two bounds:
\begin{equation}
\label{eqn:hyp}
 d_{\HH^2_\a}(\overline \G_a,\overline\G_b) \prec \log (b-a),
\end{equation}
and 
\begin{equation}
\label{eqn:pinched}
d_{\T(S_\a)}(\overline \G_a,\G_a) \prec 1, \ 
d_{\T(S_\a)}(\overline \G_b,\G_b) \prec 1.
\end{equation}
The theorem would then follow from 
Equations~(\ref{eqn:product}),(\ref{eqn:hyp}),(\ref{eqn:pinched}), 
and the triangle inequality.

\subsection*{Proof of Equation~\eqref{eqn:hyp}}
We use the estimate of distance in $\HH^2_\a$ from Equation~(\ref{eqn:distance-in-H2})
in Section~\ref{sec:minskyproduct}.  Let  $\s_t=\overline \G_t$, let
$\ell_t=l_{\s_t}(\a)$, and let $s_t$ be
the Fenchel-Nielsen twist coordinate of $\a$ at $\s_t$.
By~(\ref{eqn:distance-in-H2}) we have
$$d_{\HH^2_\a}(\s_a, \s_b)\add \frac{1}{2} \log \max \Big\{ 
|s_a - s_b|^2 \ell_a \ell_b, \frac{\ell_a}{\ell_b},
 \frac{\ell_b}{\ell_a} \Big\}.$$
 We shall to show that the contribution $|s_a - s_b|^2 \ell_a \ell_b$ coming from  the twist can be neglected.
By Lemma~\ref{lem:minskytwist}, we have for any lamination
$\xi$: 
$$
|s_a -s_b| \add |tw_{\s_a}(\xi,\a)-tw_{\s_b}(\xi,\a)|.
$$
By Proposition~\ref{prop:twistcomparison} with  $\xi=\nu^+$ (or $\xi = \nu^-$), we have
\begin{align*}
|Tw_{\s_a}(\np,\a) &- Tw_{\overline q_a}(\np,\a) | \prec  \frac{1}{\ell_a},\\
|Tw_{\s_b}(\np,\a) &- Tw_{\overline q_b}(\np,\a) | \prec \frac{1}{\ell_b} .
\end{align*}

In general, if a curve $\a$ is short on a surface $\sigma$ then, by considering the restriction to $F$, we can view  $ tw_q(\nu,\a)$ as split
 into contributions coming from the flat and the expanding annuli around $\a$.   It follows from the Gauss-Bonnet theorem that in an expanding annulus, two geodesics intersect at most once. Hence the contribution to $ tw_q(\nu,\a)$
is essentially  contained in $F(\a)$, for details see~\cite{rafi2} and the proof of Lemma 5.6 in~\cite{crs}.

In the present case, there is no flat annulus in $\overline q_t$ corresponding to
$\alpha$. Hence the twistings $Tw_{\overline q_a}(\np,\a)$ and $Tw_{\overline
 q_b}(\np,\a)$ are bounded; in fact, they are at most two.
Therefore, $|s_a - s_b|^2 \ell_a \ell_b \prec 1$ and we get
$$
d_{\HH^2_\a}(\s_a,\s_b) \add \frac{1}{2} \log \max \Big\{ 
\frac{\ell_a}{\ell_b},
 \frac{\ell_b}{\ell_a} \Big\}.$$
Since $K_{\overline q_a}(\a)= K_a(\a)$ and $K_{\overline q_b}(\a)= K_b(\a)$, 
it follows from Equation~(\ref{eqn:exponential}) that
$$\frac{\ell_a}{\ell_b} \asymp \frac{\log K_b(\a)}{\log K_a(\a)}
\prec \frac{2(b-a) + \log K_a (\a)}{\log K_a(\a)} 
\leq \frac{2(b-a)}{\log M} + 1.$$
Similarly for $\ell_b/\ell_a$, we have the identical bound. 
Thus Equation~(\ref{eqn:hyp}) is proved.

\subsection*{Proof of Equation~\eqref{eqn:pinched}} This is a consequence
of the following lemma due to Minsky:
\begin{lemma}[\cite{minskyharmonic} Lemma 8.4]
\label{lem:extremal}
Let $X$ be a closed Riemann surface and $Y \subset X$ an
incompressible subsurface. There exists a constant $m$ depending on
the topology of $X$ only, such that if each component of $\dd Y$ bounds
an annulus in $Y$ of modulus at least $m$, then for any non-peripheral simple 
closed curve
$\zeta \subset Y$,
$$\ext_Y(\zeta) \mul \ext_X(\zeta).$$
\end{lemma}
Here $\ext_Y(\zeta)$ denotes the extremal length of a curve $\zeta$ on the surface $Y$. Note that although the lemma is stated for closed surfaces, the proof works
for surfaces with punctures as well.

Continuing the proof of Equation~\eqref{eqn:pinched}, for $t=a,b$ we claim that 
\begin{equation}
\label{eqn:extremalG}
\ext_{\G_t}(\zeta) \mul \ext_{\overline \G_t}(\zeta)
\end{equation} for every
non-peripheral simple closed curve $\zeta$ in $S\setminus \a$.  Note
that if the maximal flat annulus $F_t(\a)$ at $(S,q_t)$ has modulus
bounded above by $m$, then there is a $k$-quasi-conformal homeomorphism from
$(S,q_t)$ to $(S,\overline q_t)$,
% that is Lipschitz outside $E_t(\a) \cup F_t(\a)$
 where $k$ depends only on $m$. This automatically implies
that $d_{\ts}(\G_t,\overline \G_t) \prec 1.$

Now suppose that $\mod F_t(\a)> m$. In
order to apply Lemma~\ref{lem:extremal}, we take the following
intermediate step illustrated in Figure~\ref{fig:cutandreglue}. As usual,  $E_t(\a)$ is an expanding annulus of maximal modulus around $\a$. One component of $\dd F_t(\a)$ is the inner
boundary $\dd_0$ of $E_t(\a)$.  Let $\dd_0'$ be the other component of $\dd
F_t(\a)$ and let $A_t$ be the flat annulus contained in
$F_t(\a)$ that shares $\dd_0'$ as a boundary component and that has modulus
$m$. 
%%%%%%%%%%%%%%%%%%%%
\begin{figure}[tb]
\begin{center}
\centerline{\includegraphics{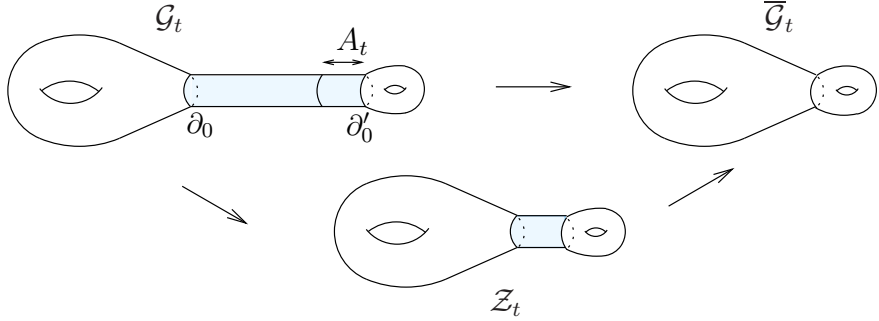}}
%\centerline{\includegraphics{figure2.pdf}}
\caption{Cut out flat annulus and re-glue.}
 \label{fig:cutandreglue}
\end{center}
\end{figure}
%%%%%%%%%%%%%%%%
Let $\cal Z_t$ be the surface which is obtained from $\G_t$ by cutting
out $F_t(\a) \setminus A_t$ and re-gluing the
boundary components together, shown as the lower surface in Figure~\ref{fig:cutandreglue}. Each
boundary component of $\cal Z_t\setminus \dd_0$ has an annulus, namely
$E_t(\a)$ and $A_t$, around it whose modulus is at least
$m$. Therefore, we can apply Lemma~\ref{lem:extremal} to $\cal
Z_t\setminus \dd_0$ as a subsurface of $\cal G_t$ and as a subsurface
of $\cal Z_t$ to obtain
$$\ext_{\G_t}(\zeta) \mul \ext_{\mathcal Z_t \setminus \dd_0}(\zeta)
\mul \ext_{\mathcal Z_t}(\zeta).
$$
Because the modulus of $A_t$ is bounded above by $m$, we have
as above that $\cal Z_t$ and $\overline \G_t$ are $k$-quasi-conformal
so that in particular,
$$\ext_{\mathcal Z_t}(\zeta) \mul \ext_{\overline \G_t}(\zeta).$$ 
This proves Equation~(\ref{eqn:extremalG}).

Thus it follows from Kerckhoff's formulation of \Teich distance
\cite{kerckasym} and Minsky's product regions theorem
that
\begin{equation*}
d_{\T(S_\a)}(\G_t, \overline \G_t) \add \frac{1}{2} \sup_{\zeta \in {\cal C} 
(S\setminus \a)} \log \frac{\ext_{\G_t}(\zeta)}{\ext_{\overline \G_t}(\zeta)}
\prec 1, 
\end{equation*}
completing the proof of Equation~\eqref{eqn:pinched}. 
\end{proof}

%%%%%%%%%%%%%%%%%%%%%
\section{The Main Theorem}
\label{sec:main}
%%%%%%%%%%%%%%%%%%%%%%%%%
 In this section we prove our main result,  Theorem~\ref{thm:qgeo0}.
To estimate $d_{\ts}(\L_a,\L_b)$ we will apply Minsky's
product regions theorem and verify in turn upper and lower bounds on the distance. We start with a lemma which will be used to estimate 
 $d_{\HH^2_\a}(\L_v,\L_w)$, where $\a$ is 
curve which is short  along an interval $[v,w]$.
 
Recall from Proposition~\ref{prop:aroundbalancetime} that
$I_\a =I_\a(\ep) $ is the maximal open interval around $t_\a$ such that 
$l_{\G_t}(\a) < \ep$ for all $t \in I_\a$. 
It follows from Theorem~\ref{thm:GLshort} that 
if a curve is sufficiently short in $\G_t$, then it is, in the coarse sense,
at least as short in $\L_t$. 
In particular, we may 
choose $\ep = \ep_1$  in Proposition~\ref{prop:aroundbalancetime}
 small enough that if $l_{\G_t}(\a) < \ep_1$, then $l_{\L_t}(\a)<\ep_0$. 
\begin{lemma}
\label{lem:HLdistance}
Let $[v,w] \subset I_\a(\ep_1)$.
\begin{itemize}
\item[(i)] If
$D_t(\a) \geq \sqrt{K_t(\a)}$ for all $t \in [v,w]$,
then $$d_{\HH^2_\a}(\L_v,\L_w) \add w-v.$$
\item[(ii)] If $\sqrt{K_t(\a)} \geq D_t(\a)$ for all $t \in
  [v,w]$, then  
$$d_{\HH^2_\a}(\L_v,\L_w) \ladd \frac{w-v}{2}.$$
\end{itemize}
\end{lemma}
\begin{proof} The proof rests on the formula~\eqref{eqn:distance-in-H2} from Section~\ref{sec:minskyproduct} 
and a careful comparison of rates of change of lengths and twists.
Let $\ell_t=l_{\L_t}(\a)$, and let
$s_t$ be the Fenchel-Nielsen twist of $\a$ at $\L_t$. As in~\eqref{eqn:distance-in-H2} we have
\begin{equation} 
\label{eqn:hyp-distance}
d_{\HH^2_\a}(\L_v,\L_w) \add \frac{1}{2} \log \max 
\Big\{ |s_v - s_w|^2 \ell_v \ell_w,
\frac{\ell_v}{\ell_w}, \frac{\ell_w}{\ell_v} \Big\}.
\end{equation}
By Lemma~\ref{lem:minskytwist}, we have
\begin{equation}
\label{eqn:xi}
|s_v -s_w| \add |tw_{\L_v}(\nu^\pm,\a)-tw_{\L_w}(\nu^\pm,\a)|.
\end{equation}

First suppose $t_{\a} \leq v < w$. Then by Theorem~\ref{thm:Ltwist}
$$
|s_v -s_w|^2 \ell_v \ell_w \,{\stackrel{{}_\ast}{\prec}}\, 
\Big[ \frac{1}{\ell_v} + \frac{1}{\ell_w} \Big]^2 \ell_v \ell_w \mul
 \max \Big\{ \frac{\ell_v}{\ell_w}, \frac{\ell_w}{\ell_v} \Big\}.
$$
Therefore,
\begin{equation*}
d_{\HH^2_\a}(\L_v,\L_w) \add
 \frac{1}{2} \log \max \Big\{
\frac{\ell_v}{\ell_w}, \frac{\ell_w}{\ell_v} \Big\}.
\end{equation*}
If
$D_t(\a) \geq \sqrt{K_t(\a)}$ for all $t \in [v,w]$,
so that $1/l_{\L_t}(\a) \mul D_t(\a)$ on $[v,w]$, then
$$\max \Big\{\frac{\ell_v}{\ell_w}, \frac{\ell_w}{\ell_v} \Big\}
\mul e^{2(w-v)}.$$
 If $\sqrt{K_t(\a)} \geq D_t(\a)$ for all $t \in
  [v,w]$, so that $1/l_{\L_t}(\a) \mul \sqrt{K_t(\a)}$ on $[v,w]$, then by Equation~(\ref{eqn:exponential}) and Lemma~\ref{lem:Kdecay}, we get
$$\sqrt{K_w(\a)} \lmul \sqrt{ K_v(\a)}
\lmul e^{w-v} \sqrt{K_w(\a)},$$
from which it follows that
$$\max \Big\{\frac{\ell_v}{\ell_w}, \frac{\ell_w}{\ell_v} \Big\}
\stackrel{{}_\ast}{\prec} e^{w-v}$$
and the lemma is proved in this case.  The case where $v < w \leq t_{\a}$ 
can be handled similarly.

Now, suppose $v<t_\alpha < w$ and for convenience, 
translate so that $t_\alpha=0$. 
If $\sqrt{K_t(\a)} \geq D_t(\a)$ for all $t \in
  [v,w]$, the result follows from the triangle inequality
  $$ d_{\HH^2_\a}(\L_v,\L_w) \leq  d_{\HH^2_\a}(\L_v,\L_0) + d_{\HH^2_\a}(\L_0,\L_w)$$
and the result already proved above.

The interesting case is that in which 
$D_t(\a) \geq \sqrt{K_t(\a)}$ for all $t \in [v,w]$, in which case 
$l_{\L_t}(\a)$ decreases on $[v,0]$ but then increases again  on $[0,w]$.
This means that
$ \max\left\{ \ell_v/\ell_w ,\ell_w/\ell_v \right\}
\mul e^{2|v+w|}$
and consequently the term  $ \frac{1}{2}
 \log \max \left\{ \ell_v/\ell_w ,\ell_w/\ell_v \right\}$
does not 
reflect the total distance $w-v$.  Instead, we have to look more carefully at
the term $|s_v - s_w|^2 \ell_v \ell_w$.

We have
$$
|s_v - s_w|  \ell_v \add | tw_{\L_v}(\nm,\a)- tw_{\L_w}(\nm,\a)|  \ell_v \add  Tw_{\L_w}(\nm,\a) \ell_v 
$$
where the second equality follows from Theorem~\ref{thm:Ltwist}. 
Similarly, 
$$
|s_v - s_w|  \ell_w \add | tw_{\L_v}(\np,\a)- tw_{\L_w}(\np,\a)|  \ell_w \add  
Tw_{\L_v}(\np,\a) \ell_w. 
$$

On the other hand, writing $d_\a = \da$ for the relative twist as defined in Section~\ref{sec:background}:
$$
d_{\a}  \ell_v \add | tw_{\L_v}(\nm,\a)- tw_{\L_v}(\np,\a)|  \ell_v \add  Tw_{\L_v}(\np,\a) \ell_v 
$$
and 
$$
d_{\a}  \ell_w \add | tw_{\L_w}(\nm,\a)- tw_{\L_w}(\np,\a)|  \ell_w \add  Tw_{\L_w}(\nm,\a) \ell_w 
$$
where we again made two applications of Theorem~\ref{thm:Ltwist}. 

Also note that by definition, $D_0(\a) = d_{\a}$ so by Theorem~\ref{thm:GLshort}
we have 
$1/\ell_0 \mul d_{\a}$. Thus
$$
|s_v - s_w|^2  \ell_v \ell_w \mul d_{\a}^2 \ell_v \ell_w \mul \frac{\ell_v}{\ell_0} \frac{\ell_w}{\ell_0}=
e^{2(w-v)}.
$$
It follows by~\eqref{eqn:hyp-distance} that
$d_{\HH^2_\a}(\L_v,\L_w) \add w-v.$
 \end{proof}

Before proving our main theorem, we also establish the following rather technical 
lemma, which quantifies more precisely the schematic graphs in Figure~\ref{fig:graphs}:
\begin{lemma}
\label{lem:dichotomy}
 Let $M$ be chosen as in Theorem~\ref{thm:Klarge}.  Then there exists $\ep >0$, depending only on the topology of $S$, such that for any $a,b$ with  $[a,b] \subset I_\a(\ep) $ 
either:
\begin{itemize}
\item[(i)] $K_t(\a) > M$ on $[a,b]$; 
\vspace*{0.2cm}
\item[ or (ii)] (i) fails, $D_t(\a) \geq \sqrt{K_t(\a)}$ on a subinterval of the form $[a,u]$, and  $$\sqrt{K_b(\a)} \lmul e^{u-a};$$ 
\item[ or (iii)] (i) fails, $D_t(\a) \geq \sqrt{K_t(\a)}$ on a subinterval of the form   $[u,b]$, and  $$\sqrt{K_a(\a)} \lmul e^{b-u}.$$
\end{itemize}
 \end{lemma}
\begin{proof} By Theorem~\ref{thm:GLshort}, we can choose $\ep>0$ 
  small enough
that if $t \in I_{\a}(\ep)$ and 
if $\sqrt{K_t(\a)} \geq D_t(\a)$ then $K_t(\a) > M$.  Thus if (i) fails, we must have 
$D_w(\a) >  \sqrt{K_w(\a)}$ for some  $w \in (a,b)$.

Suppose first that $D_t(\a) >  \sqrt{K_t(\a)}$ on $[a,b]$, and that  (i) fails, so that   there
is some $c \in [a,b]$ where $K_c(\a) \leq M$.  To check   (ii)  holds,  we have only to verify its final statement.  Since $M$ is fixed, it follows from
Equation~\eqref{eqn:exponential} that
$$ \sqrt{K_b(\a)} \lmul  e^{b-c} \sqrt{K_c(\a)} \lmul e^{b-a}.$$
(By the same  argument, (iii) also holds in this case.)

Now suppose that $D_u(\a) = \sqrt{K_u(\a)}$ for some $u \in (a,b)$.
We claim   that if  $t_\a \notin [a,b]$ then (i) holds.
Suppose for definiteness that $t_{\a} < a$. By Lemma~\ref{lem:Kdecay} we have $K_t(\a) \gmul K_u(\a)$ on $[a,u]$,
and by Lemma~\ref{lem:DandK} we have $\sqrt{K_t(\a)} \gmul D_t(\a)$ on $[u,b]$. Hence
 $1/l_{\L_t}(\a) \asymp \sqrt{K_t(\a)}$ on $[u,b]$.
Therefore, reducing  $\ep>0$ if necessary, 
we can again ensure   $K_t(\a) >M$ on
$[a,b]$ and (i) holds as claimed.

Suppose now that $D_u(\a) = \sqrt{K_u(\a)}$ for some $u \in [a,b]$ and that  $t_\a \in [a,b]$, say for definiteness that
 $t_\a < u$.  
If there is another point $u' \in
[a,t_\a]$ such that $D_{u'}(\a) = \sqrt{K_{u'}(\a)}$, then again with a suitable adjustment of $\ep$ we have
$K_t(\a) > M$ on  $[a,b]$ (see Figure~\ref{fig:graphs}) and we are in case (i). 
If  there is no such point
$u'$, then $D_t(\a) \geq \sqrt{K_t(\a)}$ on $[a,u]$. Assuming that  in addition (i) fails, there is a point $c \in [a,b]$
where $K_c(\a) \leq M$. 
By Lemma~\ref{lem:DandK} we have $1/l_{\L_t}(\a) \mul
\sqrt{K_t(\a)}$ on $[u,b]$. Assuming $\ep$ is sufficiently
small, we deduce that $c \in [a,u]$. Then
by Lemma~\ref{lem:Kdecay} we have
$$
\sqrt{K_c(\a)} \gmul e^{-|t_\a - c|} \sqrt{K_{t_\a}(\a)} 
\gmul e^{-|t_\a - c|} \sqrt{K_b(\a)} 
\geq e^{-(u-a)} \sqrt{K_b(\a)} $$
and we are in case (ii).   
The case where $t_\a > u$ is handled similarly and results in
  (iii).
\end{proof}

%%%%%%%%%%%%%%%%%%%%%%%%%%%%%%%%%%%%%%%%%%%
We are now ready to prove our main result Theorem~\ref{thm:qgeo0}.

\medskip

  \noindent \textbf{Proof of Theorem~\ref{thm:qgeo0}.} As noted in the introduction, we prove the theorem by obtaining separate upper and lower bounds for $d_{\T(S)}(\L_a,\L_b)$. The upper bound is relatively straightforward but the lower bound requires an inductive procedure based on Lemma~\ref{lem:dichotomy}.
  
In order to compare two surfaces $\L_a,\L_b$  at the ends of a long interval $[a,b] \subset \RR$, it is convenient to consider separately the curves which are short at $a$ but not at $b$, those which 
are short at  $b$ but not at $a$, and those which are short at both. More precisely, 
choose $\ep$ to satisfy Lemma~\ref{lem:dichotomy} and then choose $\ep'\leq \ep$ as in Proposition~\ref{prop:aroundbalancetime}
so that    if $ l_{\G_t}(\a) < \ep'$ then $t \in I_\a (\ep)$.
In particular, if $l_{\G_a}(\a) < \ep'$ and $l_{\G_b}(\a) < \ep'$, then
since $I_\a(\ep)$ is connected,
$[a,b] \subset I_\a(\ep)$.  
Now define subsets $\Gamma_a, \Gamma_b$, and $\Gamma$ of the curves of length less than $\ep'$ in either $\G_a$ or $\G_b$ as
 follows:
\begin{align}
\label{eqn:Gammas}
 \Gamma_a  &= \{ \a \in \cal C (S) : l_{\G_a}(\a) < \ep',
\ l_{\G_b}(\a) \geq \ep' \}, \nonumber \\
 \Gamma_b  &= \{ \a \in \cal C (S) : l_{\G_b}(\a) < \ep',
\ l_{\G_a}(\a)\geq \ep' \},\\
 \Gamma \hspace{0.1cm} &= \{ \a \in \cal C (S) : l_{\G_t}(\a) < \ep',
\ \mbox{ for } t =a,b\}.  \nonumber 
\end{align}

We begin by establishing some preliminary estimates on distances in the \Teich spaces of the  subsurfaces obtained by cutting along these curves.
By Minsky's product regions theorem,
\begin{align*}
d_{\T(S_\Gamma)}(\L_a,\G_a) &\add \max_{\a \in \Gamma_a}
[\, d_{\T(S_{\Gamma \cup \Gamma_a})}(\L_a, \G_a),
 d_{\HH^2_\a}(\L_a,\G_a) ],\\
d_{\T(S_\Gamma)}(\L_b,\G_b) &\add \max_{\a \in \Gamma_b}
[\, d_{\T(S_{\Gamma \cup \Gamma_b})}(\L_b, \G_b),
\, d_{\HH^2_\a}(\L_b,\G_b) ].
\end{align*}
Now on the one hand, by Theorem~\ref{thm:GLdistance},
the thick parts  
of $\G_a$ 
and $\L_a$ are bounded distance from one another, as are the 
those of $\G_b$ and $\L_b$. Therefore
$$d_{\T(S_{\Gamma \cup \Gamma_a})}(\L_a, \G_a) \prec 1
\quad\text{and}\quad
d_{\T(S_{\Gamma \cup \Gamma_b})}(\L_b, \G_b) \prec 1.$$
On the other hand, because the twisting is bounded as in 
Theorems~\ref{thm:gtwist} and \ref{thm:Ltwist}, we have
\begin{align*}
d_{\HH_\a^2}(\L_a, \G_a) &\ladd 
\frac{1}{2} \log \frac{l_{\G_a}(\a)}{l_{\L_a}(\a)} <  \frac12 \log \frac 1{l_{\L_a}(\a)}
\quad\text{for}\quad \a \in \Gamma_a, \\
d_{\HH_\a^2}(\L_b, \G_b) &\ladd 
\frac{1}{2} \log \frac{l_{\G_b}(\a)}{l_{\L_b}(\a)} <
\frac12 \log \frac 1{l_{\L_b}(\a)}
\quad\text{for}\quad \a \in \Gamma_b.
\end{align*} 
Thus it follows that
%% \begin{align}
%% \begin{split} 
%% \label{eqn:GLab}
%% d_{\T(S_\Gamma)}(\L_a, \G_a) &\ladd  \frac12 \max_{\a \in \Gamma_a} \,
%% \log \frac 1{l_{\L_a}(\a)}, \\
%% d_{\T(S_\Gamma)}(\L_b, \G_b) &\ladd  \frac12 \max_{\a \in \Gamma_b} \,
%% \log \frac 1{l_{\L_b}(\a)}.
%% \end{split}
%% \end{align}
\begin{equation}
\label{eqn:GLab}
\begin{array}{c}
\displaystyle
d_{\T(S_\Gamma)}(\L_a, \G_a) \ladd  \frac12 \max_{\a \in \Gamma_a} \,
\log \frac 1{l_{\L_a}(\a)}, \\[5mm]
\displaystyle
d_{\T(S_\Gamma)}(\L_b, \G_b) \ladd  \frac12 \max_{\a \in \Gamma_b} \,
\log \frac 1{l_{\L_b}(\a)}.
\end{array}
\end{equation}

We now turn to bounding the distance $d_{\T(S)}(\L_a, \L_b)$.
By Minsky's product regions theorem, 
\begin{equation}
\label{eqn:product2}
d_{\T(S)}(\L_a,\L_b) \add 
\max _{\a \in \Gamma} \Big[ d_{\T(S_\Gamma)}(\L_a,\L_b), 
d_{\HH^2_\a}(\L_a,\L_b) \Big].
\end{equation}

\subsection*{Upper bound } 
We prove the upper bound $d_{\ts}(\L_a,\L_b) \ladd 3(b-a)$ by bounding the terms on the right hand side of ~\eqref{eqn:product2}.

By  Lemma~\ref{lem:HLdistance} we have
$d_{\HH^2_\a}(\L_a,\L_b) \ladd b-a$ for each $\a \in \Gamma$.
We provide an upper bound for $d_{\T(S_\Gamma)}(\L_a,\L_b)$
using the triangle inequality
\begin{equation}
\label{eqn:triangle}
d_{\T(S_\Gamma)}(\L_a,\L_b) \leq d_{\T(S_\Gamma)}(\L_a,\G_a) +
d_{\T(S_\Gamma)}(\G_a,\G_b) + d_{\T(S_\Gamma)}(\L_b,\G_b).
\end{equation}

To bound the first and last terms of the right hand side, we will use~\eqref{eqn:GLab} and the fact that the length $l_{\L_t}(\a)$
of a curve increases at rate at most $e^{2t}$. More precisely,
notice that
if $\a \in \Gamma_a$ then $I_\a \cap [a,b] = [a,c)$ for some $c \leq b$.
By definition of $I_\a$, we have $l_{\G_c}(\a)=\ep$.
Then it follows from Theorem~\ref{thm:GLshort} that 
$l_{\L_c}(\a)$ is bounded below by a uniform constant
that depends only on $\ep$.
Therefore, by the observation following Equation~\eqref{eqn:exponential}, we have
$$
\log \frac{1}{l_{\L_a}(\a)} \ladd  \log \frac{l_{\L_c}(\a)}{l_{\L_a}(\a)}
 \ladd 2(b-a). 
$$
Similarly, if $\a \in \Gamma_b$, then
$$
\log \frac{1}{l_{\L_b}(\a)} \ladd 2(b-a). 
$$
Therefore, from~\eqref{eqn:GLab} it follows that
$$d_{\T(S_\Gamma)}(\L_a, \G_a) \ladd b-a \quad\text{and}\quad
 d_{\T(S_\Gamma)}(\L_b, \G_b) \ladd b-a.
 $$
The second term in~\eqref{eqn:triangle} is bounded by
Minsky's product regions theorem:
$$d_{\T(S_\Gamma)}(\G_a,\G_b) \ladd d_{\T(S)}(\G_a,\G_b) = b-a.$$ 
This finishes the proof of the upper bound.

 \subsection*{Lower bound}

We prove the lower bound  $d_{\ts}(\L_a,\L_b) \gadd (b-a)/4$ by showing that at least one of the terms in the right hand side
of~\eqref{eqn:product2} is bounded below by $(b-a)/4$.

We begin by reducing the problem to a consideration of the curves in $\Gamma$ only.
It follows from a theorem of Wolpert  \cite{wolpert} that for every $\gamma
\in \cal C (S)$,
$$d_{\ts}(\L_a,\L_b) \geq \frac{1}{2} \Bigg| \log \frac{l_{\L_b}(\gamma)}{
l_{\L_a}(\gamma)}\Bigg|.$$
It follows as before from Theorem~\ref{thm:GLshort} 
that if $l_{\G_t}(\a) \geq \ep'$, then the length $l_{\L_t}(\a)$ 
is uniformly bounded below. 
Therefore, we have:
%% \begin{eqnarray}
%% \begin{split}
%% \label{eqn:gammaab}
%%  \Bigg| \log \frac{l_{\L_b}(\a)}{l_{\L_a}(\a)} \Bigg| &\gadd
%% \log \frac{1}{l_{\L_a}(\a)} \mbox{ for } \a \in \Gamma_a, \\
%%  \Bigg| \log \frac{l_{\L_a}(\a)}{l_{\L_b}(\a)} \Bigg| &\gadd
%% \log \frac{1}{l_{\L_b}(\a)} \mbox{ for } \a \in \Gamma_b.
%% \end{split}
%% \end{eqnarray}
\begin{equation}
\label{eqn:gammaab}
\begin{array}{c}
\displaystyle
 \Bigg| \log \frac{l_{\L_b}(\a)}{l_{\L_a}(\a)} \Bigg| \gadd
\log \frac{1}{l_{\L_a}(\a)} \mbox{ for } \a \in \Gamma_a, \\[5mm]
\displaystyle
 \Bigg| \log \frac{l_{\L_a}(\a)}{l_{\L_b}(\a)} \Bigg| \gadd
\log \frac{1}{l_{\L_b}(\a)} \mbox{ for } \a \in \Gamma_b.
\end{array}
\end{equation}
It follows from the triangle inequality that if either
$$\max_{\a \in \Gamma_a}\log \frac{1}{l_{\L_a}(\a)} \geq \frac{b-a}{2}
\ \mbox{ or }
\max_{\a \in \Gamma_b} \log \frac{1}{l_{\L_b}(\a)} \geq \frac{b-a}{2},$$
the lower bound is proved.  

Thus we may assume that 
\begin{equation}
\label{eqn:assumption}
 \max_{\a \in \Gamma_a}\log \frac{1}{l_{\L_a}(\a)} \leq  \frac{b-a}{2}
\ \mbox{ and }
\max_{\a \in \Gamma_b} \log \frac{1}{l_{\L_b}(\a)} \leq \frac{b-a}{2},
\end{equation}
bringing us to the key part of the proof.
From Minsky's product region theorem, we have 
\begin{equation}
\label{eqn:productG}
 b-a=d_{\T(S)}(\G_a,\G_b) \add
 \max _{\a \in \Gamma}[ d_{\T(S_\Gamma)}(\G_a,\G_b), d_{\HH^2_\a}(\G_a,\G_b)].
\end{equation}
We claim that either $d_{\T(S_\Gamma)}(\G_a,\G_b) \add b-a$, or that
there is some $\a \in \Gamma$ such that $d_{\HH^2_\a}
(\G_a,\G_b) \add b-a$
and such that Lemma~\ref{lem:dichotomy} (ii) or (iii) holds. After proving the claim, we will show that either alternative implies the required bound on $d_{\ts}(\L_a,\L_b) $.
We are going to use an inductive argument for which it is important to note than we can choose the additive constant in Minsky's product
regions theorem to be fixed for all surfaces obtained from $S$ by cutting out
 any 
subset of curves in $\Gamma$.

If the maximum in~\eqref{eqn:productG} is
realized by $d_{\T(S_\Gamma)}(\G_a,\G_b)$, then
obviously $d_{\T(S_\Gamma)}(\G_a,\G_b) \add b-a$. 
 If the maximum in~\eqref{eqn:productG} is realized by $d_{\HH^2_\gamma}
 (\G_a,\G_b)$
for some $\gamma \in \Gamma$,
consider the alternatives for $\gamma$ in Lemma~\ref{lem:dichotomy}.  
If (i) holds, 
then by  Theorem~\ref{thm:Klarge} we have
$d_{\T(S_\gamma)}(\G_a,\G_b) \add b-a$.
In this case, we apply Minsky's product
regions theorem to $S_\gamma$, giving
\begin{equation}
\label{eqn:product3}
 b-a \add d_{\T(S_\gamma)}(\G_a,\G_b) \add
 \max _{\delta \in \Gamma \setminus \gamma}[ d_{\T(S_\Gamma)}
 (\G_a,\G_b), d_{\HH^2_\delta}(\G_a,\G_b)].
\end{equation}

Now  repeat the same argument; if the maximum in~\eqref{eqn:product3}
is realized by 
$d_{\HH^2_\delta}(\G_a,\G_b)$ for some $\delta \in \Gamma \setminus \gamma$
that satisfies Lemma~\ref{lem:dichotomy} (i), then apply the
product regions theorem to $S_{\{\gamma,\delta\}}$.  Eventually, up to a finite number of changes to the additive constants, 
either there must be some
$\a \in \Gamma$ for which $d_{\HH^2_\a}(\G_a, \G_b) \add b-a$ such that
Lemma~\ref{lem:dichotomy} (i) does not hold, or it must be that $d_{\T(S_\Gamma)}(\G_a,\G_b) \add b-a$.
The claim follows.

Now we show that either alternative implies the required bound. 
If $d_{\T(S_\Gamma)}(\G_a,\G_b) \add b-a$, then
the triangle inequality, Equation~\eqref{eqn:GLab},
and the assumption~\eqref{eqn:assumption} give
\begin{align*}
d_{\T(S_\Gamma)}(\L_a,\L_b)  &\gadd 
d_{\T(S_\Gamma)}(\G_a,\G_b) - d_{\T(S_\Gamma)}(\L_a,\G_a) -
d_{\T(S_\Gamma)}(\L_b,\G_b) \\
&\gadd d_{\T(S_\Gamma)}(\G_a,\G_b) -\frac{1}{2} \max_{\a \in \Gamma_a}
\log \frac{1}{l_{\L_a}(\a)}
-\frac{1}{2} \max_{\a \in \Gamma_b} \log \frac{1}{l_{\L_b}(\a)}
\\
&\add \frac{b-a}{2}. 
\end{align*}

Now assume the alternative: that there is some $\a \in \Gamma$ such that
$d_{\HH^2_\a}(\G_a, \G_b) \add b-a$ and
such that Lemma~\ref{lem:dichotomy} (ii) or (iii) holds. 
Assume (ii) holds: we have 
that $D_t(\a) \geq \sqrt{K_t(\a)}$ on an interval
$[a,u]$ and consider the following  two cases depending on the length of
$[a,u]$. (Case (iii) can be handled similarly.)

 If
$u-a \geq (b-a)/2$, then the triangle inequality and 
Lemma~\ref{lem:HLdistance} give
\begin{align*}
d_{\HH^2_\a}(\L_a,\L_b) &\geq d_{\HH^2_\a}(\L_a,\L_u) - d_{\HH^2_\a}(\L_u,\L_b)
\\ &\gadd (u-a) - \frac{b-u}{2} \\
& \geq \frac{b-a}{4}.
\end{align*}
(Strictly speaking, it may be that $\sqrt{K_t(\a)} < D_t(\a)$ for some values of
$t \in [u,b]$. However, Lemma~\ref{lem:DandK} implies that
$1/l_{\L_t}(\a) \mul \sqrt{K_t(\a)}$ on $[u,b]$, and this is sufficient to guarantee
that $d_{\HH^2_\a}(\L_u,\L_b) \ladd (b-u)/2 $,
see the proof of Lemma~\ref{lem:HLdistance}).  

If $u-a < (b-a)/2$, then consider the
triangle inequality 
$$d_{\HH^2_\a}(\L_a,\L_b) \geq d_{\HH^2_\a}(\G_a,\G_b) -
d_{\HH^2_\a}(\G_a,\L_a) - d_{\HH^2_\alpha}(\G_b,\L_b).$$
Similarly to our previous argument, since the twisting is bounded as in Theorems~\ref{thm:gtwist}
and~\ref{thm:Ltwist}, we have 
$$
d_{\HH^2_\a}(\G_a,\L_a) \ladd \frac{1}{2} \log \frac{l_{\G_a}(\a)}{l_{\L_a}(\a)}
\quad\text{and}\quad
d_{\HH^2_\a}(\G_b,\L_b) \ladd \frac{1}{2} \log \frac{l_{\G_b}(\a)}{l_{\L_b}(\a)}.
$$
Since $D_a(\a) \geq \sqrt{K_a(\a)}$ it follows from Theorem~\ref{thm:GLshort} that
$$\log \frac{l_{\G_a}(\a)}{l_{\L_a}(\a)} \mul 1.$$ 
Since Lemma~\ref{lem:dichotomy} (ii) holds,
it follows from  the assumption 
$u-a < (b-a)/2$ that
$$
\log \frac{l_{\G_b}(\a)}{l_{\L_b}(\a)} <
\log \frac{1}{l_{\L_b}(\a)} \ladd u-a 
< \frac{b-a}{2}. 
$$
Thus, in this case we have
$$d_{\HH^2_\a}(\L_a,\L_b) \gadd \frac{3}{4} (b-a).$$
This concludes the proof. \hfill $\Box$
%\end{proof}
%\qed
 
%%%%%%%%%%%%%%%%%%%%%%%%%%%%%%%%%%%%%%%%%%%%%%%%%%%%%%%%%%%%%%%%%%

%%%%%%%%%%%%%%%%%%%%%%%%%%%%%%%%%%%%%%%%%%%%%%%%%%%%%%%%

%%%%%%%%%%%%%%%%%%%%%%%%%%%%%%%%%%%%%%%%%%%%%%%%%%%%%%%%%%%%%%%%%%%%%%
\end{document}